\documentclass[11pt]{article}
\def\thetitle{Small subgraphs in the trace of a random walk}

\usepackage{amsmath,amssymb}

\usepackage[usenames,dvipsnames,svgnames,table]{xcolor}
\definecolor{CombinatoricaAqua}{HTML}{00698C}
\definecolor{CombinatoricaBlue}{HTML}{3A3293}
\definecolor{CombinatoricaBrown}{HTML}{66220C}
\definecolor{CombinatoricaRed}{HTML}{DF2A27}
\definecolor{HarvardCrimson}{rgb}{0.6471, 0.1098, 0.1882}

\makeatletter
\let\reftagform@=\tagform@
\def\tagform@#1{\maketag@@@
  {(\ignorespaces\textcolor{CombinatoricaBrown}{#1}\unskip\@@italiccorr)}}
\renewcommand{\eqref}[1]{\textup{\reftagform@{\ref{#1}}}}
\makeatother

\usepackage[backref=page]{hyperref}
\hypersetup{
  unicode,
  pdfencoding=auto,
  pdfauthor={Michael Krivelevich and Peleg Michaeli},
  pdftitle={\thetitle},
  pdfsubject={05C81 (Primary), 05C80, 60G50 (Secondary)},
  pdfkeywords={random walk, small subgraphs},
  colorlinks=true,
  citecolor=CombinatoricaBlue,
  linkcolor=CombinatoricaAqua,
  anchorcolor=CombinatoricaBrown,
  urlcolor=HarvardCrimson}

\usepackage{amsthm}
\usepackage{thmtools}

\usepackage[capitalize]{cleveref}


\declaretheoremstyle[
  spaceabove=\topsep, spacebelow=\topsep,
  headfont=\color{CombinatoricaBrown}\normalfont\bfseries,
  bodyfont=\itshape,
]{thm}
\declaretheoremstyle[
  spaceabove=\topsep, spacebelow=\topsep,
  headfont=\color{CombinatoricaBrown}\normalfont\bfseries,
  bodyfont=\normalfont,
]{dfn}
\declaretheoremstyle[
  spaceabove=0.5\topsep, spacebelow=0.5\topsep,
  headfont=\color{CombinatoricaBrown}\normalfont\bfseries,
  bodyfont=\normalfont,
]{rmk}

\declaretheorem[style=thm,parent=section]{theorem}
\declaretheorem[style=thm,sibling=theorem]{lemma}
\declaretheorem[style=thm,sibling=theorem]{corollary}
\declaretheorem[style=thm,sibling=theorem]{claim}
\declaretheorem[style=rmk,sibling=theorem]{remark}
\declaretheorem[style=remark,numbered=no]{note}
\declaretheorem[style=definition,numbered=no]{acknowledgement}

\usepackage[nobysame,msc-links]{amsrefs}

\BibSpec{book}{%
    +{}  {\PrintPrimary}                {transition}
    +{,} { \textbf}                     {title} 
    +{.} { }                            {part}
    +{:} { \textit}                     {subtitle}
    +{,} { \PrintEdition}               {edition}
    +{}  { \PrintEditorsB}              {editor}
    +{,} { \PrintTranslatorsC}          {translator}
    +{,} { \PrintContributions}         {contribution}
    +{,} { }                            {series}
    +{,} { \voltext}                    {volume}
    +{,} { }                            {publisher}
    +{,} { }                            {organization}
    +{,} { }                            {address}
    +{,} { \PrintDateB}                 {date}
    +{,} { }                            {status}
    +{}  { \parenthesize}               {language}
    +{}  { \PrintTranslation}           {translation}
    +{;} { \PrintReprint}               {reprint}
    +{.} { }                            {note}
    +{.} {}                             {transition}
    +{}  {\SentenceSpace \PrintReviews} {review}
}

\makeatletter
\renewcommand{\PrintNames@a}[4]{%
  \PrintSeries{\name}
    {#1}
    {}{ and \set@othername}
    {,}{ \set@othername}
    {}{ and \set@othername}
    {#2}{#4}{#3}%
}
\makeatother

\usepackage{sectsty}
\sectionfont{\color{CombinatoricaBrown}}
\subsectionfont{\color{CombinatoricaBrown}}
\subsubsectionfont{\color{CombinatoricaBrown}}

\usepackage{soul}
\soulregister\ref7

\usepackage[textsize=tiny,backgroundcolor=black!10]{todonotes}
\presetkeys{todonotes}{noline}{}

\usepackage[a4paper]{geometry}
\title{\thetitle}
\author{Michael Krivelevich
\thanks{School of Mathematical Sciences, Raymond and Beverly Sackler Faculty
  of Exact Sciences, Tel Aviv University, Tel Aviv 6997801, Israel. E-mail:
  \href{mailto:krivelev@post.tau.ac.il}{\tt krivelev@post.tau.ac.il}. Research supported in part by a USA-Israel
BSF Grant and by a grant from Israel Science Foundation.}
\and Peleg Michaeli
\thanks{School of Mathematical Sciences, Raymond and Beverly Sackler Faculty
  of Exact Sciences, Tel Aviv University, Tel Aviv 6997801, Israel. E-mail:
  \href{mailto:peleg.michaeli@math.tau.ac.il}
  {\tt peleg.michaeli@math.tau.ac.il}.}}

\def\td{\rho}
\def\md{m_0}
\input{commands.tex}

\begin{document}
\maketitle

\begin{abstract}
  We consider the combinatorial properties of the trace of a random walk on the
  complete graph and on the random graph $\gnp$.  In particular, we study the 
  appearance of a fixed subgraph in the trace.  We prove that for a subgraph 
  containing a cycle, the threshold for its appearance in the trace of a random 
  walk of length $m$ is essentially equal to the threshold for its appearance 
  in the random graph drawn from $\gnm$.  In the case where the base graph is 
  the complete graph, we show that a fixed forest appears in the trace 
  typically much earlier than it appears in $\gnm$.
\end{abstract}

\section{Introduction}\label{introduction}
For a positive integer $n$ and a real $p\in[0,1]$, we denote by $\gnp$ the
probability space of all (simple) labelled graphs on the vertex set
$[n]=\left\{1,\ldots,n\right\}$, where every pair of vertices is connected 
independently with probability $p$.
A closely related model, which we denote by $\gnm$,
is the \emph{uniform} probability space over all graphs on $n$ vertices with
$m$ edges.  Both models have been extensively studied since first introduced
by Gilbert~\cite{Gil59}, and by \erdos{} and \renyi{}~\cites{ER_rgI,ER_evo}.

One of the problems studied in~\cite{ER_evo} was the problem of finding the 
threshold for the appearance of a fixed subgraph.  Formally, given a fixed 
graph $H$, one is interested in the smallest value of $p_0$ such that when 
$p\gg p_0$ the random graph $\gnp$ contains a copy of $H$ \emph{with high 
probability} (\whp{}), that is, with probability tending to $1$ as $n$ grows.
It turns out that the threshold for the appearance of $H$ is determined by 
$\md(H)$, the maximum edge density of all of its non-empty subgraphs.  In 
symbols,
\begin{equation*}
  \md(H) =
  \max\left\{
    \frac{\left|E\left(H'\right)\right|}
    {\left|V\left(H'\right)\right|}
    \midd
    H'\subseteq H,\ \left|V\left(H'\right)\right|>0
  \right\}.
\end{equation*}

The problem of finding the threshold for every fixed subgraph was settled by
\bollobas{}~\cite{Bol81rg} in 1981, and the result can be stated as follows
(see also~\cite{AS4}*{Section 4.4} or~\cite{JLR}*{Theorem 3.4}).

\begin{theorem}\label{gnp:subgraphs}
  Let $H$ be a fixed non-empty graph and let $G\sim\gnp$.  Then,
  \begin{equation*}
    \lim_{n\to\infty}\pr{H\subseteq G} =
    \begin{cases}
      0 & p \ll n^{-1/\md(H)}\\
      1 & p \gg n^{-1/\md(H)}.
    \end{cases}
  \end{equation*}
\end{theorem}

\begin{theorem}\label{gnm:subgraphs}
  Let $H$ be a fixed non-empty graph and let $G\sim G(n,m)$.  Then,
  \begin{equation*}
    \lim_{n\to\infty}\pr{H\subseteq G} =
    \begin{cases}
      0 & m \ll n^{2-1/\md(H)}\\
      1 & m \gg n^{2-1/\md(H)}.
    \end{cases}
  \end{equation*}
\end{theorem}

Here and later, the notation $f\gg g$ means that $f/g\to\infty$.
For a vertex $v$, denote by $N(v)$ the set of its neighbours, and let $N^+(v)=
\left\{v\right\}\cup N(v)$.
Given a (finite) base graph $G=(V,E)$, a (lazy) \emph{simple random walk} on 
$G$ is a stochastic process $\left(X_0,X_1,\ldots\right)$ where $X_0$ is 
sampled uniformly at random from $V$, and for $t\ge 0$, $X_{t+1}$ is sampled 
uniformly at random from $N^+\left(X_t\right)$, independently of the past.  The 
\emph{trace} of the random walk at time $t$ is the (random) subgraph 
$\Gamma_t\subseteq G$ on the same vertex set, whose edges consist of all edges 
traversed by the walk by time $t$, excluding loops and suppressing possible 
edge multiplicity.  Formally,
\begin{equation*}
  E\left(\Gamma_t\right)
  = \left\{\left\{X_{s-1},X_s\right\}\mid 0<s\le t,\ X_{s-1}\ne X_s\right\}.
\end{equation*}

\begin{note}
There are various definitions of laziness of random walks, perhaps the most 
common is staying put with probability $1/2$ (see, e.g., \cite{LPW}); however, 
for the case of random walks on the complete graph on $n$ vertices, a random 
walk which stays put with probability $1/n$ yields an independent sequence of 
uniformly distributed locations, which is far easier to handle.  We decided
therefore to adopt here a general definition of laziness which, in the case of 
the complete graph, behaves like that.  However, as the thresholds discussed in 
this work are coarse, the results below can be applied for more traditional 
definitions of laziness, as well as for non-lazy random walks.
\end{note}

In~\cite{BL13} it was shown that the trace of a random walk whose length is 
proportional to $n^2$ on (dense) \emph{quasirandom} graphs (including dense 
random graphs) on $n$ vertices is typically quasirandom.
In~\cite{FKMP}, several results were given concerning
graph-theoretic properties of the trace, for sparser base graphs and shorter
random walks.
In this paper we continue this study of the structure of the trace, finding
thresholds for the appearance of fixed subgraphs.
Our first result, which is analogous to \cref{gnm:subgraphs},
considers the random walk on the random graph $\gnp$, and is
restricted to fixed subgraphs containing a cycle.  As we will see later,
that restriction is necessary, as the statement is simply false for forests.

Note that the condition $\md(H)\ge 1$ is equivalent to the condition of
containing a cycle.

\begin{theorem}\label{gnp:trace:subgraphs}
  Let $H$ be a fixed graph with $\md(H)\ge 1$, let
  $\varepsilon>0$, ${p\ge n^{-1/\md(H)+\varepsilon}}$ and $G\sim\gnp$, and let
  $\Gamma_t$ be the trace of a random walk of length $t$ on $G$.
  Then,
  \begin{equation*}
    \lim_{n\to\infty} \pr{H\subseteq \Gamma_t}
    =\begin{cases}
      0 & t \ll n^{2-1/\md(H)}\\
      1 & t \gg n^{2-1/\md(H)}.
    \end{cases}
  \end{equation*}
\end{theorem}

\begin{remark}\label{pseudorandom}
When proving the above theorem, we do not really require that $G$ is random, but
rather that it possesses some pseudo-random properties, which occur with high
probability in $\gnp$.
\end{remark}

The complementary case $\md(H)<1$ is in fact quite different, and we were able
to find the threshold in that case for random walks on the complete graph $K_n$ 
only.  We will discuss potential difficulties in this aspect in 
\cref{section:concluding}.  Denote by $\odd(G)$ the number of odd degree 
vertices in $G$.

\begin{theorem}\label{trace:subtrees}
  Let $T$ be a fixed tree on at least $2$ vertices with $\odd(T)=\theta$.
  Let $\Gamma_t$ be the trace of a random walk of length $t$ on $K_n$.
  Then,
  \begin{equation*}
    \lim_{n\to\infty} \pr{T\subseteq \Gamma_t}
    =\begin{cases}
      0 & t \ll n^{1-2/\theta}\\
      1 & t \gg n^{1-2/\theta}.
    \end{cases}
  \end{equation*}
\end{theorem}

In particular, the theorem implies that the probability that the trace contains 
a fixed path (the case $\theta=2$) as a subgraph is $1-o(1)$ if $t\gg 1$.
The corollary below follows easily from \cref{trace:subtrees}.

\begin{corollary}\label{trace:subforests}
  Let $F$ be a non-empty fixed forest, and let $T_1,\ldots,T_z$ be its
  connected components.  Let $\theta=\max_{i\in[z]}\left\{\odd(T_i)\right\}$.
  Let $\Gamma_t$ be the trace of a random walk of length $t$ on $K_n$.  Then,
  \begin{equation*}
    \lim_{n\to\infty} \pr{F\subseteq \Gamma_t}
    =\begin{cases}
      0 & t \ll n^{1-2/\theta}\\
      1 & t \gg n^{1-2/\theta}.
    \end{cases}
  \end{equation*}
\end{corollary}

The overall proof strategy of \cref{gnp:trace:subgraphs,trace:subtrees} is to 
apply the first and the second moment methods.  Our key lemma 
(\cref{key_lemma}) estimates the probability that the random walk on a random 
graph will traverse the edges of a \emph{fixed} copy of a constant-sized graph 
$H$.  We find that if $t\gg n$, the probability for the appearance of a copy in 
the trace is asymptotically equivalent to the probability of its 
appearance in a uniform random choice of a subgraph of $\gnp$ with $t$ edges, 
and if $t\ll n$, it is determined by a structural property of $H$, namely, by 
the smallest number $\td$ for which $H$ admits a trail decomposition with $\td$ 
parts.  For the proof of the key lemma we use standard tools from Markov chain 
theory, and, in particular, a result about the mixing time of random graphs.

The rest of the paper is organized as follows.  In \cref{section:gnp} we state 
the key lemma and present some preliminary results to be used in its proof.  
The lemma itself is proved in \cref{section:key_lemma:proof}, and in 
\cref{proof:gnp:trace:subgraphs} we use it to prove 
\cref{gnp:trace:subgraphs}.  \Cref{section:kn} contains the proofs of 
\cref{trace:subtrees,trace:subforests}.  Finally, in \cref{section:concluding}, 
we conclude with some remarks and open problems.

\section{Walking on \texorpdfstring{$\gnp$}{G(n,p)}}\label{section:gnp}
Recall that a \emph{walk} on $G$ is a sequence of vertices $v_1,\ldots,v_t$
such that for $1\le i<t$, $\left\{v_i,v_{i+1}\right\}$ is an edge of $G$, and
that a \emph{trail} on $G$ is a walk in which all of these edges are
distinct.  Denote by $\td(G)$ the \emph{trail decomposition number} of $G$,
that is, the minimum number of edge-disjoint trails in $G$ whose union is the 
edge set of $G$.

We begin with a key lemma.
In what follows, we use $\mathbb{P}$ to denote the probability given that the
initial distribution of the walk is uniform, and $\mathbb{P}_\mu$ to denote the
probability given that the initial distribution is $\mu$.

\begin{lemma}\label{key_lemma}
  Let $\varepsilon,\gamma>0$, $p\ge n^{-1+\varepsilon}$, $G\sim\gnp$ and
  $p^{-1} \ll t = O\left(n^{2-\gamma}p\right)$.
  Let $H$ be a fixed
  graph with $\ell\ge 1$ edges and $\td\left(H\right)=\td$.  Then,
  \whp{} (over the distribution of $G$), for a fixed copy $H_0$ of $H$ in $G$,
  \begin{equation*}
    \pr{H_0\subseteq \Gamma_t\mid G} =
    \Theta\left(\left(np\right)^{-\ell}\sum_{r=\td}^\ell 
    \left(\frac{t}{n}\right)^r \right).
  \end{equation*}
  Moreover, if $t\gg n$, then
    \begin{equation*}
      \pr{H_0\subseteq \Gamma_t\mid G} =
      \left(\frac{2t}{n^2p}\right)^\ell (1+o(1)).
    \end{equation*}
\end{lemma}

The assumption that $p^{-1}\ll t = O\left(n^{2-\gamma}p\right)$ in the statement
of the lemma is artificial.  The upper bound on $t$ is essential for proving 
(in \cref{gnp:multiplicity}) that the random walk traverses all edges at most a 
constant number of times with very high probability -- a fact which is clearly 
not true for every $t$.
The lower bound on $t$ is used to show that it is ``too expensive'' for the walk
to traverse an edge of $H_0$ more than once (see \eqref{eq:lambda2}).  As we 
will see later, these bounds on $t$ do not affect the proofs of our main 
theorems.

Before proving the lemma, we state a simple corollary.

\begin{corollary}\label{key_corollary}
  Let $H$ be a fixed graph with $k$ vertices, $\ell\ge 1$
  edges, $\md(H)=\md$ and $\td(H)=\td$.  Let $\varepsilon,\gamma>0$,
  $\nu=\max\{\md,1\}$, $p\ge n^{-1/\nu+\varepsilon}$, $G\sim\gnp$ and
  $p^{-1} \ll t = O\left(n^{2-\gamma}p\right)$.
  Finally, let $Z$ be a random variable counting the number of copies of
  $H$ in $\Gamma_t$ (where multiple edges are ignored).  Then, \whp{} (over the 
  distribution of $G$),
  \begin{equation*}
    \E{Z\mid G} = \Theta\left(
    n^{k-\ell}\sum_{r=\td}^{\ell}\left(\frac{t}{n}\right)^r \right).
  \end{equation*}
\end{corollary}

\begin{proof}[Proof (of the corollary)]
  Since $p\ge n^{-1/\nu+\varepsilon}\gg n^{-1/\md}$, the number of copies of
  $H$ in $G$ is \whp{} asymptotically equal to its expectation
  (see for example~\cite{JLR}*{Remark 3.7})
  which is $\Theta\left(n^kp^\ell\right)$.
  The result then follows from \cref{key_lemma} and the linearity of
  expectation.
\end{proof}

Our goal now is to prove \cref{key_lemma}.  In what follows,
$\varepsilon,\gamma>0$ are fixed constants, $p\ge n^{-1+\varepsilon}$,
$G\sim\gnp$, $X_0,X_1,\ldots,X_t$ is a (lazy, simple) random walk on $G$
starting at a uniformly chosen vertex, $\Gamma_t$ is its trace and
$p^{-1}\ll t=O\left(n^{2-\gamma}p\right)$.
The \emph{transition rate} of $X$ from $u$ to $v$ is the probability
\begin{equation*}
  p_{uv} = \pr{X_{t+1}=v\mid X_t=u} =
  \pr{X_1=v\mid X_0=u},
\end{equation*}
and for an integer $s\ge 0$ we denote
\begin{equation*}
  p_{uv}^s = \pr{X_{t+s}=v\mid X_t=u} =
  \pr{X_s = v\mid X_0 = u}.
\end{equation*}
Since, as is well known, $G$ is \whp{} connected,
the sequence $X$ forms an irreducible Markov chain, hence it has a unique
stationary distribution given by (see, e.g., \cite{LPW}*{Section 1.5})
\begin{equation*}
  \pi_v = \frac{d(v)}{\sum_{u\in [n]}d(u)} = \frac{d(v)}{2\left|E\right|}.
\end{equation*}

The following lemma about the degree distribution in $\gnp$ can easily be proved
using standard estimates for the tail of the binomial distribution.

\begin{lemma}\label{gnp:regular}
  With high probability, $d(v)\sim np$, and thus $\pi_v\sim n^{-1}$, for every
  $v\in[n]$.
\end{lemma}

We will use the fact that the random walk on $\gnp$ ``mixes well''.
Roughly speaking, this means that the walk quickly forgets its starting point,
and the distribution of its location quickly approaches stationarity.
Recall that the \emph{total variation distance} between the distribution of 
$X_t$ and the stationary distribution is
\begin{equation*}
  \dtv{X_t}{\pi} = \frac{1}{2}\sum_{v\in[n]}\left|\pr{X_t=v}-\pi_v\right|.
\end{equation*}
In \cite{Hil96}, Hildebrand showed\footnote{Hildebrand shows this for a 
non-lazy random walk.  However, as the probability that the lazy walk stays
put at least once in a walk of fixed length is $o(1)$, we may ignore this 
difference here.} that there exists a constant $s=s(\varepsilon)$ for which, 
\whp{} (and regardless of the starting distribution),
\begin{equation*}
  \dtv{X_s}{\pi} < 1/e.
\end{equation*}
It follows (see, i.e., \cite{LPW}*{Section 4.5}) that for an integer $\ell>0$,
\begin{equation*}
  \dtv{X_{\ell s}}{\pi} < (2/e)^\ell.
\end{equation*}
We therefore obtain the following.

\begin{claim}\label{gnp:mixing}
  For every $x>0$ there exists $B=B(\varepsilon,x)=O(\ln{n})$ such that \whp{}
  \begin{equation*}
    \dtv{X_B}{\pi} = o(n^{-x}).
  \end{equation*}
\end{claim}

Let $x$ be a large positive constant to be determined later.
Say that a vertex distribution $\pi'$ is \emph{almost stationary} if
$\dtv{\pi'}{\pi} = o\left(n^{-x}\right)$.
The last corollary practically means that regardless of the starting 
distribution, after $B$ steps, say, the distribution of the walk is almost 
stationary.

For a vertex $v$, let
$\mathbf{n}_v$ be the uniform distribution over $N(v)$, and for $s>0$ denote
by $\eta(v,s)$ the number of exits the walk has made from vertex $v$ by time
$s$.  Formally,
\begin{equation*}
  \eta(v,s) = \left|\left\{i\in[s]\mid X_{i-1}=v,\ X_i\ne v\right\}\right|.
\end{equation*}
A key observation is that typically no vertex is visited too many times, hence
no edge is traversed too many times.  This is stated in the following two
lemmas.

\begin{lemma}\label{gnp:exits}
  For every $\alpha>0$ there exists $\gamma'>0$ such that \whp{} (over the 
  distribution of $G$), the probability that the random walk (of length $t$) 
  visits at least one of the vertices more than $n^{1-\gamma'}p$ times is 
  $o\left(n^{-\alpha}\right)$.
\end{lemma}

\begin{proof}
  First note that we may assume that $\gamma\le\varepsilon$; otherwise, let
  $t_\varepsilon=n^{2-\varepsilon}p\gg n^{2-\gamma}p = \Omega(t)$.
  We can now prove the lemma for a walk of length $t_\varepsilon$, and conclude
  that the result holds for the walk of length $t$.

  Fix $v\in[n]$ and let $s=n^{1-\gamma/2}$.  Observe that in order to exit
  $v$, starting at a vertex which is not $v$, the walk must first enter it,
  and in view of \cref{gnp:regular} the probability for that to happen
  at any given step is $O\left(1/\left(np\right)\right)$.
  It follows that \whp{} (over the distribution of $G$),
  \begin{equation*}
    Q:=\pr_{\mathbf{n}_v}{\eta(v,B)\ge 1\mid G} = O\left(\frac{B}{np}\right)
    = O\left(\frac{\ln{n}}{n^\varepsilon}\right)
    = o\left(n^{-\gamma/2}\right).
  \end{equation*}
  For an integer $a>0$, let
  \begin{equation*}
    P_\mu(a) := \pr_{\mu}{\eta(v,s)\ge a\mid G}.
  \end{equation*}
  Note that for an almost stationary distribution $\pi'$, and for large enough 
  $x$, by the union bound we have that \whp{}
  \begin{equation*}
    P_{\pi'}(1) \le P_\pi(1) + o\left(n^{-x}\right) = O(s/n)
    = O\left(n^{-\gamma/2}\right),
  \end{equation*}
  and for $a>1$, there exists an almost stationary distribution $\pi''$ for 
  which
  \begin{align*}
    P_{\pi'}(a) \le P_\pi(a) + o\left(n^{-x}\right)
    &\le P_\pi(a-1)\left(Q+P_{\pi''}(1)\right) + o\left(n^{-x}\right)\\
    &= P_\pi(a-1)\cdot O\left(n^{-\gamma/2}\right) + o\left(n^{-x}\right),
  \end{align*}
  as the probability of visiting $v$ at least $a$ times is at most the 
  probability of visiting it $a-1$ times, and conditioning on that, the
  probability of visiting it once more, which is at most the probability of
  visiting it during the first $B$ steps after exiting from it, plus the 
  probability of visiting it at least once during $s$ steps, starting from
  (another) almost stationary distribution $\pi''$.  By induction,
  for $a>2(\alpha+2)/\gamma$ and $x>a\gamma/2$,
  \begin{equation}
    P_{\pi'}(a) \le P_\pi(1)\cdot O\left(n^{-(a-1)\gamma/2}\right)
    + o\left(n^{-x}\right)
    = O\left(n^{-a\gamma/2}\right)
    = o\left(n^{-\alpha-2}\right).\label{eq:avisits}
  \end{equation}
  Now, let
  \begin{equation*}
    L=\ceil{t/(s+B)} = O\left(n^{1-\gamma/2}p\right) = o(n).
  \end{equation*}
  Consider
  dividing $[t]$ into $L$ segments of length at most $s$, with ``buffers'' of
  length $B$ between them. Noting that the distribution of the first vertex is
  uniform (hence almost stationary),
  it follows from~\eqref{eq:avisits} and the union 
  bound that (\whp{} over the distribution of $G$) with probability
  $o\left(n^{-\alpha-1}\right)$ there exists a segment in which the walk exits
  $v$ at least $a$ times.  Considering the possible visits in the buffers 
  between the segments as well (at most $BL$ such visits), we conclude that 
  with probability $o\left(n^{-\alpha-1}\right)$ the walk exits $v$ more than 
  $n^{1-\gamma'}p$ times by time $t$, for $\gamma'=\gamma/3$, say.
  The union bound over all vertices yields the desired result.
\end{proof}

\begin{lemma}\label{gnp:multiplicity}
  For every $\alpha>0$ there exists $M>0$ such that \whp{} (over the 
  distribution of $G$), the probability that the random walk (of length $t$) 
  traverses at least one of the edges more than $M$ times is 
  $o\left(n^{-\alpha}\right)$.
\end{lemma}

\begin{proof}
  For a vertex $v$ and integer $i\ge 0$, let $x_v^i\sim\mathbf{n}_v$,
  independently of each other.  Think of the random walk $X_t$ as follows.
  $X_0$ is sampled uniformly at random from $V$, and at each time
  $t\ge 0$, $X_{t+1}$ is determined as follows: with probability
  $1/\left(d\left(X_t\right)+1\right)$ it equals $X_t$, and with the
  remaining probability it equals $x_{X_t}^{\eta\left(X_t,t\right)}$.
  We think of $x_v^i$ as being sampled before the walk is performed, and the 
  walk, when it exits $v$ for the $i$'th time, simply reveals 
  $x_v^i$\footnote{This is somewhat similar to the \emph{list model} described
  in \cite{BL13}.}.

  Let $(u,v)$ be a directed edge.  Let $x_{uv}^i$ be the indicator of the
  event $x_u^i=v$.  The number of traversals of $(u,v)$ during the first
  $\eta$ exits from $u$ is therefore (\whp{}) the sum of $\eta$ independent
  Bernoulli-distributed random variables with success probability (roughly)
  $1/(np)$.  Thus, the probability that $(u,v)$ was traversed at least
  $M$ times during the first $\eta$ exits from $u$ equals the probability
  that a binomial random variable with $\eta$ trials and success probability
  (roughly) $1/(np)$ is at least $M$.  The probability that $(u,v)$ was
  traversed at least $M$ times is at most the probability that it was
  traversed at least $M$ times during the first $\eta$ exits from $u$ in
  addition to the probability that the walk has exited $u$ more than $\eta$
  times.

  Thus, by the union bound, the probability that there exists $(u,v)$ which
  was traversed at least $M$ times by time $t$ is at most
  \begin{equation*}
    n^2\cdot\pr{\bin{\eta}{\frac{(1+o(1))}{np}}\ge M}
    + \pr{\exists u:\ \eta(u,t)>\eta}.
  \end{equation*}
  Choosing $\eta=2n^{1-\gamma'}p$, with the right $\gamma'$,
  \cref{gnp:exits} tells us that the second term is $o\left(n^{-\alpha}\right)$,
  and standard concentration results for the binomial distribution tell us that
  for large enough $M$
  the first term is $o\left(n^{-\alpha}\right)$, concluding the proof.
\end{proof}

For a set $W\subseteq[t]$ denote by $r(W)$ the minimum number of integer
intervals whose union is $W$.  In symbols,
\begin{equation*}
  r(W) = \left|\left\{1\le i\le t\mid i\in W \wedge i+1\notin W\right\}\right|.
\end{equation*}
For $W$ with $r(W)=r$ write
\begin{equation*}
  W = \left\{t_1,t_1+1,\ldots,t_1+a_1-1,
             t_2,t_2+1,\ldots,t_2+a_2-1,
             t_3\ldots,
             t_r,t_r+1,\ldots,t_r+a_r-1\right\},
\end{equation*}
where $t_i-1\notin W$ for $i\in[r]$ and $t_i+a_i<t_j$ for $1\le i<j\le r$.
If $t_{i+1}-(t_i+a_i)<3B$, we say that the $(i+1)$'th run is \emph{defective},
and we denote by $q(W)=\left|\left\{i\in[r-1]\mid
t_{i+1}-(t_i+a_i)<3B\right\}\right|$
the number of defective runs in $W$.  Let
\begin{equation*}
  \mathcal{W}_{w,r} =
  \left\{W\subseteq[t]\mid |W|=w,\ r(W)=r\right\},
\end{equation*}
and
\begin{equation*}
  \mathcal{W}_{w,r,q} =
  \left\{W\subseteq[t]\mid |W|=w,\ r(W)=r,\ q(W)=q\right\}.
\end{equation*}

\begin{claim}\label{ws:count}
  For every $1\le r\le w$,
  \begin{equation*}
    \left|\mathcal{W}_{w,r}\right| =
    \binom{w-1}{r-1}\binom{t-w+1}{r}.
  \end{equation*}
\end{claim}

\begin{proof}
  For every $\mathbf{a}=(a_i)_{i=1}^r$ with $a_i>0$ and $\sum_{i=1}^r a_i=w$,
  let $\mathcal{W}_\mathbf{a}$ be the set of $W$'s in $\mathcal{W}_{w,r}$ with
  run lengths $a_1,\ldots,a_r$.  The number of $W$'s in $\mathcal{W}_\mathbf{a}$
  is the number of ways to locate $r$ runs with lengths $a_1,\ldots,a_r$ in
  $[t]$ so that any two distinct runs will be separated by at least $1$.  For
  every $\mathbf{a}$, this number is the number of integer solutions to 
  the equation
  \begin{equation*}
    \sum_{i=0}^r b_i = t-w,\qquad
    \begin{cases}
      b_0,b_r\ge 0\\
      b_i\ge 1 & 1\le i\le r-1,
    \end{cases}
  \end{equation*}
  where we think of $b_0$ as the space before the first run, $b_r$ the space
  after the last run, and for $1\le i\le r-1$, $b_i$ is the space between the
  $i$'th run and the one following it.  Thus
  \begin{equation*}
    \left|\mathcal{W}_\mathbf{a}\right| = \binom{t-w+1}{r}.
  \end{equation*}
  Since the number of $\mathbf{a}$'s with $a_i>0$ and $\sum_{i=1}^r a_i=w$ is
  the number of integer solutions to the equation
  \begin{equation*}
    \sum_{i=1}^r a_i = w,\qquad
    \forall 1\le i\le r,\ a_i>0,
  \end{equation*}
  it follows that
  \begin{equation*}
    \left|\mathcal{W}_{w,r}\right| =
    \binom{w-1}{r-1}\binom{t-w+1}{r}.
  \end{equation*}
\end{proof}

\begin{lemma}\label{scattered}
  Let $K>0$ be fixed, let $r\le w\le K$ and suppose $t\gg 1$.
  Sample $W$ uniformly from $\mathcal{W}_{w,r}$.  Then,
  \begin{equation*}
    \pr{q(W)\ge q} = O\left(\left(Bt^{-1}\right)^q\right).
  \end{equation*}
\end{lemma}

\begin{proof}
  Given a set $J\subseteq[r-1]$ with $|J|=q$, $I=[r-1]\smallsetminus J$
  and $\mathbf{b}=(b_j)_{j\in J}$ with $1\le b_j < 3B$ for $j\in [q]$,
  let $A_{J,\mathbf{b}}$ be the set of $W\in\mathcal{W}_{w,r}$ for which
  for every $j\in J$, $t_{j+1}-(t_j+a_j)=b_j$.
  The cardinality of
  $A_{J,\mathbf{b}}$ is the number of solutions to the integer equation
  \begin{equation*}
    b_0+b_r+\sum_{i\in I} b_i = t-w-\sum_{j\in J}b_j,\qquad
    \begin{cases}
      b_0,b_r\ge 0\\
      b_i\ge 1 & i\in I,
    \end{cases}
  \end{equation*}
  which is clearly at most the number of integer solutions to the equation
  \begin{equation*}
    b_0+b_r+\sum_{i\in I} b_i = t,\qquad
    \begin{cases}
      b_0,b_r\ge 0\\
      b_i\ge 1 & i\in I.
    \end{cases}
  \end{equation*}
  It was shown in \cref{ws:count} that $\left|\mathcal{W}_{w,r}\right|
  =\Theta\left(t^r\right)$.  By a similar argument,
  $\left|A_{J,\mathbf{b}}\right| = O\left(t^{r-q}\right)$.
  The union bound over all choices of $J$ and $\mathbf{b}$ yields
  \begin{equation*}
    \pr{q(W)\ge q}
    \le
    \binom{r-1}{q} (3B)^q \cdot
    \frac{O\left(t^{r-q}\right)}{\Theta\left(t^r\right)}
    = O\left(\left(Bt^{-1}
    \right)^q\right).
  \end{equation*}
\end{proof}

For $i\in [t]$ let $e_i=\left\{X_{i-1},X_i\right\}$ and let
$\vec{e}_i = \left(X_{i-1},X_i\right)$.
For a fixed subgraph $H$ of $G$ let $W(H)\subseteq [t]$ be the (random) set
of times in which an edge from $H$ had been traversed.  That is,
\begin{equation*}
  W(H) = \left\{i\in[t]\mid e_i \in E(H)\right\}.
\end{equation*}

We are now ready to prove our key lemma.

\subsection{Proof of \texorpdfstring{\cref{key_lemma}}{Lemma \ref{key_lemma}}}
\label{section:key_lemma:proof}
Let $\varepsilon,\gamma>0$, $p\ge n^{-1+\varepsilon}$, $G\sim\gnp$ and
$p^{-1}\ll t=O\left(n^{2-\gamma}p\right)$.  As promised in 
\cref{pseudorandom}, we assume that $G$ possesses the properties guaranteed 
\whp{} by \cref{gnp:regular,gnp:mixing,gnp:multiplicity}.
Let $H$ be a fixed graph with $\ell\ge 1$ edges, $k$ vertices and 
$\td(H)=\td$, and let $H_0$ be a copy of $H$ in $G$.
Let $A$ be the event $H_0\subseteq\Gamma_t$, and for any $W\subseteq[t]$ let
$A_W$ be the event $A\land\left(W\left(H_0\right)=W\right)$.  
Our goal now is to estimate $\pr{A}$.

\begin{claim}\label{W:props}
  If $\pr{A_W}$ is positive then
  \begin{itemize}
    \item $\ell\le |W|\le t$,
    \item $1\le r(W)\le |W|$,
    \item $0\le q(W) < r(W)$, and
    \item $r(W)\ge \ell + \td - |W|$.
  \end{itemize}
\end{claim}

\begin{proof}
  The only non-obvious claim is that $r(W)\ge \ell + \td - |W|$.  We will prove
  it by decomposing $H_0$ into at most $|W|+r(W)-\ell$ trails.  Suppose
  $W_1,\ldots,W_r$ are the $r=r(W)$ runs of $W$, and let $w_1,\ldots,w_r$ be
  their lengths.  Let $\ell_i$ be the number of edges of $H_0$ that were 
  traversed by $W_i$ but not by $W_j$ for $j<i$.  By removing from $W_i$ every 
  edge that was previously traversed by either $W_i$ or by an earlier run, we 
  create at most $1+(w_i-\ell_i)$ edge-disjoint trails, which are disjoint to 
  every trail created so far.  At the end of this process we have created at 
  most
  \begin{equation*}
    \sum_{i=1}^r (1+w_i-\ell_i) = r+|W|-\ell
  \end{equation*}
  edge-disjoint trails covering $H$.
\end{proof}

As a result of \cref{W:props}, letting $r_w=\max\left\{1,\ell+\td-w\right\}$,
we have:
\begin{equation}\label{eq:pa}
  \pr{A} = \sum_{w=\ell}^t \sum_{r=r_w}^w \sum_{q=0}^{r-1}
  \sum_{W\in\mathcal{W}_{w,r,q}} \pr{A_W}.
\end{equation}

\subsubsection*{Upper bound}
Let $M>0$ be such that the probability that any edge was traversed at least
$M$ times is $o\left(n^{-3\ell}\right)$, as guaranteed by
\cref{gnp:multiplicity}, and let $K=\ell M$.  Write
  
\begin{equation*}
  \Lambda_{w,r,q} = \sum_{W\in\mathcal{W}_{w,r,q}}\pr{A_W},\qquad
  \Lambda_{w,r} = \sum_{q=0}^{r-1} \Lambda_{w,r,q},\qquad
  \Lambda_{w,r}^+ = \Lambda_{w,r} - \Lambda_{w,r,0},
\end{equation*}
and
\begin{equation*}
  \Lambda_1 = \sum_{w=K}^t\sum_{r=r_w}^w \Lambda_{w,r},\qquad
  \Lambda_2 = \sum_{w=\ell+1}^{K-1}\sum_{r=r_w}^w \Lambda_{w,r},\qquad
  \Lambda_3 = \sum_{r=\td}^\ell \Lambda_{\ell,r},
\end{equation*}
so, noting that $r_\ell=\td$ it follows from \eqref{eq:pa} that
\begin{equation}\label{eq:lambdas}
  \pr{A} = \Lambda_1+\Lambda_2+\Lambda_3.
\end{equation}
Now, according to the choice of $K$, we have that
\begin{equation}\label{eq:lambda1}
  \Lambda_1 \ll n^{-3\ell} \ll
  (np)^{-\ell} \sum_{r=\td}^\ell\left(\frac{t}{n}\right)^r.
\end{equation}
Let $W\in\mathcal{W}_{w,r,q}$ with $w<K$.  In these settings,
\begin{equation}\label{eq:key:AW}
  \pr{A_W} \le \pr{W\subseteq W\left(H_0\right)}
  = O\left( n^{-r+q}\left(np\right)^{-w-q}\right),
\end{equation}
as at the beginning of any non-defective run the probability that the walk
will be at a vertex of $H_0$ is $\Theta\left(1/n\right)$ (and there are $r-q$
non-defective runs), at the beginning of any defective run the probability
that the walk will be at a vertex of $H_0$ is
$O\left(1/\left(np\right)\right)$, and at any time of $W$, the probability
that the walk will traverse an edge of $H_0$ is
$O\left(1/\left(np\right)\right)$.

If $w<K$, it follows from \cref{ws:count,scattered} that
\begin{equation}\label{eq:Wwrq}
  \left|\mathcal{W}_{w,r,q}\right| = O\left(\frac{t^rB^q}{t^q}\right),
\end{equation}
and therefore it follows from \eqref{eq:key:AW} and since $B\ll tp$, that
\begin{align*}
  \Lambda_{w,r}^+
  &= \sum_{q=1}^{r-1}
  O\left(t^r\left(Bt^{-1}\right)^q n^{-r+q}\left(np\right)^{-w-q}\right)\\
  &= O\left( (np)^{-w}\left(\frac{t}{n}\right)^r
  \sum_{q=1}^{r-1}\left(\frac{B}{tp}\right)^q
  \right)
  \ll (np)^{-w}\left(\frac{t}{n}\right)^r,
  \numberthis \label{eq:lambda:defective}
\end{align*}
and
\begin{equation*}
  \Lambda_{w,r,0} = O\left((np)^{-w}\left(\frac{t}{n}\right)^r\right),
\end{equation*}
and therefore
\begin{equation}\label{eq:lambdawr}
  \Lambda_{w,r} = O\left((np)^{-w}\left(\frac{t}{n}\right)^r\right).
\end{equation}
Suppose that $\ell<w<K$.
If $t\ge n$ then, since $t\ll n^2p$ and using \eqref{eq:lambdawr},
\begin{equation*}
  \sum_{r=r_w}^w\Lambda_{w,r}
  = O\left(\left(np\right)^{-w} \left(\frac{t}{n}\right)^w\right)
  \ll \left(np\right)^{-\ell} \left(\frac{t}{n}\right)^\ell
  =\Theta\left(
  (np)^{-\ell} \sum_{r=\td}^\ell\left(\frac{t}{n}\right)^r \right),
\end{equation*}
and if $t<n$ then, since $t\gg p^{-1}$ and using \eqref{eq:lambdawr},
\begin{align*}
  \sum_{r=r_w}^w\Lambda_{w,r}
  &= O\left( (np)^{-w} \left(\frac{t}{n}\right)^{r_w} \right)\\
  &= O\left(\left(np\right)^{-\ell}\left(\frac{t}{n}\right)^\td
     \cdot \left(\frac{t}{n}\right)^{\ell-w}(np)^{\ell-w} \right)\\
  &\ll (np)^{-\ell} \left(\frac{t}{n}\right)^\td
  =\Theta\left(
  (np)^{-\ell} \sum_{r=\td}^\ell\left(\frac{t}{n}\right)^r \right),
\end{align*}
and therefore
\begin{equation}\label{eq:lambda2}
  \Lambda_2 \ll \left(np\right)^{-\ell}
  \sum_{r=\td}^\ell\left(\frac{t}{n}\right)^r.
\end{equation}
Finally, using \eqref{eq:lambdawr},
\begin{equation}\label{eq:lambda3}
  \Lambda_3 = O\left(
    (np)^{-\ell} \sum_{r=\td}^\ell\left(\frac{t}{n}\right)^r
  \right),
\end{equation}
and therefore, using \eqref{eq:lambdas}, \eqref{eq:lambda1}, 
\eqref{eq:lambda2} and \eqref{eq:lambda3},
\begin{equation*}\label{eq:pa:upper}
  \pr{A} = O\left(
    (np)^{-\ell} \sum_{r=\td}^\ell\left(\frac{t}{n}\right)^r
  \right).
\end{equation*}
This concludes the proof of the upper bound of the first part of the lemma.

\subsubsection*{Lower bound}
Let $\Gamma_W = \left\{ \{X_{i-1},X_i\} \mid i\in W \right\}$.

\begin{claim}\label{notinw}
For $W\in \mathcal{W}_{\ell,r,0}$,
\begin{equation*}
  \pr{A_W} \sim \pr{H_0\subseteq \Gamma_W}.
\end{equation*}
\end{claim}

\begin{proof}
First note that
\begin{align*}
  \pr{A_W}
  &= \pr{(W(H_0)=W)\wedge (H_0\subseteq\Gamma_W) }\\
  &= \pr{W(H_0)\subseteq W\mid H_0\subseteq\Gamma_W}
     \cdot\pr{H_0\subseteq\Gamma_W}.
\end{align*}
Now, conditioning on 
$H_0\subseteq\Gamma_W$, the probability
that an edge of $H_0$ is ever traversed during times not in $W$, can be bounded
from above as follows.  Let
\begin{equation*}
  W_B = \left\{s\in [t]\mid \exists s'\in W,\ \left|s-s'\right|\le B \right\}.
\end{equation*}
Let $\vec{e}=(u,v)$ be an arbitrary edge of $H_0$ with the direction assigned
to it.
Let $i\in [t]\smallsetminus W_B$, and assume first that $i$ is between two
consecutive runs of $W$.
Let $i_0$ be the maximal element in $W$ with
$i_0<i$, and let $i_1$ be the minimal element in $W$ with $i<i_1$.
Write $s_0=i-i_0$, $s_1=i_1-i$.
Observing that for every two vertices $v_1,v_2$ and $s>B$ we have
$p_{v_1v_2}^s \sim n^{-1}$, we have that for every $u_0,u_1$,
\begin{align*}
  \pr{X_{i-1}=u\mid X_i=v,\ X_{i_0}=u_0}
  &= \frac{p_{u_0u}^{s_0-1}p_{uv}}
          {\sum_{w\in N^+(v)} p_{u_0w}^{s_0-1}p_{wv}}\\
  &\sim \frac{p_{uv}}{\sum_{w\in N^+(v)} p_{wv}}
   \sim \frac{p_{vu}}{\sum_{w\in N^+(v)} p_{vw}}
   = p_{vu} \sim \frac{1}{np},
\end{align*}
and
\begin{equation*}
  \pr{X_i=v\mid X_{i_0}=u_0,\ X_{i_1}=u_1}
  = \frac{p_{u_0v}^{s_0}p_{vu_1}^{s_1}}
             {\sum_{w\in[n]} p_{u_0w}^{s_0}p_{wu_1}^{s_1}}
  \sim \frac{1}{n},
\end{equation*}
thus
\begin{align*}
  &\pr{\vec{e}_i=\vec{e}\mid H_0\subseteq\Gamma_W, X_{i_0}=u_0, X_{i_1}=u_1}\\
  &= \pr{X_{i-1}=u, X_i=v\mid X_{i_0}=u_0, X_{i_1}=u_1}\\
  &= \pr{X_{i-1}=u\mid X_i=v,\ X_{i_0}=u_0}
     \cdot \pr{X_i=v\mid X_{i_0}=u_0,\ X_{i_1}=u_1}
     \sim \frac{1}{n^2p}.
\end{align*}
Since this holds for every $u_0,u_1$, the probability that
$i\in W\left(H_0\right)$ is $O\left(1/\left(n^2p\right)\right)$.
Now let $i\in W_B\smallsetminus W$,
and let $i_0,i_1$ and $s_0,s_1$ be as before.
Since $W\in\mathcal{W}_{\ell,r,0}$, $s_0+s_1\ge 3B$.  Suppose first that 
$s_0>B$.  In that case,
\begin{align*}
  \pr{\vec{e}_i=\vec{e}\mid H_0\subseteq\Gamma_W, X_{i_0}=u_0, X_{i_1}=u_1}
  &= \pr{X_{i-1}=u, X_i=v\mid X_{i_0}=u_0, X_{i_1}=u_1}\\
  &\le \pr{X_{i-1}=u\mid X_i=v,\ X_{i_0}=u_0}
  \sim \frac{1}{np}.
\end{align*}

If on the other hand $s_0\le B$ then $s_1>B$ and we may use the reversibility
of the walk to obtain a similar bound for
$\pr{\vec{e}_i=\vec{e}\mid H_0\subseteq\Gamma_W}$, and therefore, since this
holds for every $u_0,u_1$, the probability that $i\in W(H_0)$ is $O(1/np)$.

If $i<\min W$ (or $i>\max W$), letting $i_1$ ($i_0$, respectively) be as 
before, a similar argument, now conditioning only on the location of $X$ at 
time $i_1$ (at time $i_0$, respectively), gives the same bounds.

Since $\left|W_B\right| = O\left(B\right)$, $B\ll np$ and $t\ll n^2p$ we 
have that
\begin{align*}
  \pr{W\left(H_0\right)\not\subseteq W
    \mid H_0\subseteq\Gamma_W}
  &= \pr{\exists i\notin W,\ i\in W(H_0) \mid H_0\subseteq\Gamma_W}\\
  &= O\left( B(np)^{-1} + t\left(n^2p\right)^{-1} \right) = o(1),
\end{align*}
and thus
\begin{equation*} 
  \pr{A_W} \sim \pr{H_0\subseteq\Gamma_W}.\qedhere
\end{equation*}
\end{proof}

Now, let $W\in \mathcal{W}_{\ell,r,0}$ with $\td\le r\le \ell$.
In this case,
\begin{equation*}
  \pr{H_0\subseteq \Gamma_W} = 
  \Omega\left( (np)^{-\ell} n^{-r}\right).
\end{equation*}
This can be seen as follows.  Let
\begin{equation*}
  f^1_1,\ldots,f^1_{\ell_1},\ldots,f^r_1,\ldots,f^r_{\ell_r}
\end{equation*}
be a decomposition of the edges of $H_0$ into $r$ trails (think of the 
edges $f^j_i$ as directed edges, with the direction induced by the $j$'th 
trail), and write $f^j_i=(u^j_i,v^j_i)$.  At the beginning of the $j$'th run of 
$W$ (which is non-defective),
the probability that the walk will be at $u^j_1$ is $\Omega(1/n)$, and the
$i$'th time in the $j$'th run of $W$, the probability that the traversed edge
is $f^j_i$, given that the location of the walk before that move is $u^j_i$, is
$\Omega(1/(np))$.  
Using \cref{notinw} we have that
\begin{equation*}
  \pr{A_W} = \Omega\left( (np)^{-\ell} n^{-r}\right).
\end{equation*}
Therefore,
\begin{equation*}
  \Lambda_{\ell,r}
  \ge \Lambda_{\ell,r,0}
  = \sum_{W\in\mathcal{W}_{\ell,r,0}}\pr{A_W}
  = \Omega\left( (np)^{-\ell} \left(\frac{t}{n}\right)^r \right),
\end{equation*}
and thus
\begin{equation*}
  \Lambda_3 = \Omega\left(
    (np)^{-\ell}\sum_{r=\td}^\ell \left(\frac{t}{n}\right)^r
  \right).
\end{equation*}
Using \eqref{eq:lambdas} we have that
\begin{equation*}
  \pr{A} = \Omega\left(
    (np)^{-\ell}\sum_{r=\td}^\ell \left(\frac{t}{n}\right)^r
  \right).
\end{equation*}
This concludes the proof of the lower bound of the first part of the lemma.

\subsubsection*{The case $t\gg n$}
In this case, according to \eqref{eq:lambda1},
\begin{equation}\label{eq:lambda1:tggn}
  \Lambda_1 \ll \left(\frac{t}{n^2p}\right)^\ell,
\end{equation}
and according to \eqref{eq:lambda2},
\begin{equation}\label{eq:lambda2:tggn}
  \Lambda_2 \ll \left(\frac{t}{n^2p}\right)^\ell.
\end{equation}

Let $W\in \mathcal{W}_{\ell,\ell,0}$.  In this case we can give a more 
accurate estimate on $\pr{A_W}$.  There are $\ell!$ ways to order the edges of 
$H_0$ by their traversal times, and for each such ordering, as all the runs are
non-defective and of length $1$, the probability 
that the walk will traverse an edge at a prescribed time is approximately the 
inverse of the number of edges in $G$.  Therefore, using \cref{notinw}, we have 
that
\begin{equation*}
  \pr{A_W} \sim \ell! \cdot \left(\frac{2}{n^2p}\right)^\ell.
\end{equation*}
According to \cref{ws:count} and \eqref{eq:Wwrq},
\begin{equation*}
  \left|\mathcal{W}_{\ell,\ell,0}\right|
  \sim \binom{\ell-1}{\ell-1} \binom{t-\ell+1}{\ell}
  = \binom{t-\ell+1}{\ell}
\end{equation*}
and thus
\begin{equation*}
  \Lambda_{\ell,\ell,0}
  \sim \binom{t-\ell+1}{\ell} \cdot \ell!
    \cdot\left(\frac{2}{n^2p}\right)^\ell
  \sim \left(\frac{2t}{n^2p}\right)^\ell.
\end{equation*}
It follows from \eqref{eq:lambda:defective} that
\begin{equation*}
  \Lambda_{\ell,\ell}^+
  \ll \left(\frac{t}{n^2p}\right)^\ell,
\end{equation*}
hence
\begin{equation*} 
  \Lambda_{\ell,\ell} 
  = \Lambda_{\ell,\ell,0} + \Lambda_{\ell,\ell}^+
  \sim \left(\frac{2t}{n^2p}\right)^\ell.
\end{equation*}
Now suppose that $\td\le r<\ell$.  It follows from \eqref{eq:lambdawr} that
\begin{equation*}
  \Lambda_{\ell,r}
  = O\left( (np)^{-\ell} \left( \frac{t}{n} \right)^r \right)
  \ll \left(\frac{t}{n^2p}\right)^\ell,
\end{equation*}
thus
\begin{equation}\label{eq:lambda3:tggn}
  \Lambda_3 \sim \Lambda_{\ell,\ell}
  \sim \left(\frac{2t}{n^2p}\right)^\ell.
\end{equation}
It follows from \eqref{eq:lambdas}, together with \eqref{eq:lambda1:tggn},
\eqref{eq:lambda2:tggn} and \eqref{eq:lambda3:tggn}, that if $t\gg n$,
\begin{equation*}
  \pr{A_W} \sim \left(\frac{2t}{n^2p}\right)^\ell,
\end{equation*}
concluding the proof of the second part of the lemma.\qed

\subsection{Proof of \texorpdfstring{\cref{gnp:trace:subgraphs}}
  {Theorem \ref{gnp:trace:subgraphs}}}
\label{proof:gnp:trace:subgraphs}
Throughout this subsection $H$ is a fixed graph with $k$
vertices, $\ell$ edges and $\md(H)=\md\ge 1$, $\varepsilon>0$, $p\ge
n^{-1/\md+\varepsilon}$ and $G$ is sampled according to $\gnp$.

\subsubsection{Proof of the negative part}
\label{proof:gnp:trace:subgraphs:negative}
Assume $t\ll n^{2-1/\md}$.
Since $p^{-1}\le n^{1/m_0-\varepsilon}\ll n\le n^{2-1/\md}$ we may 
assume without loss of generality that $t\gg p^{-1}$.
In addition, letting $\gamma\le\varepsilon$ we have that
$t=O(n^{2-\gamma}p)$.
Let $H'\subseteq H$ with $k_0$ vertices and
$\ell_0$ edges be such that $\ell_0/k_0=\md$, and write
$\td=\td\left(H'\right)$.
Let $Z,Z'$ count the number of appearances of a copy of $H,H'$ in $\Gamma_t$, 
respectively. From \cref{key_corollary} it follows that \whp{}
\begin{equation*}
  \E{Z'\mid G} = O\left(
  n^{k_0-\ell_0}\sum_{r=\td}^{\ell_0}\left(\frac{t}{n}\right)^r \right).
\end{equation*}

Now, if $\md=1$ then $k_0=\ell_0$ and $t\ll n$ and thus \whp{} 
$\E{Z'\mid G}=o(1)$.
If $\md>1$ then $k_0-\ell_0\le -1$; in that case, if $t<n$ then \whp{}
$\E{Z'\mid G}=O\left(n^{-1}\right)=o(1)$, and if $t\ge n$ we have that \whp{}
\begin{equation*}
  \E{Z'\mid G} = O\left(n^{k_0-2\ell_0}t^{\ell_0}\right)
  = o\left(n^{k_0-2\ell_0}n^{2\ell_0-k_0}\right) = o(1).
\end{equation*}
Since the non-appearance of a copy of $H'$ in $\Gamma_t$ implies that of $H$,
Markov's inequality yields the desired result.\qed

\subsubsection{Proof of the positive part}
\label{proof:gnp:trace:subgraphs:positive}
Assume $t\gg n^{2-1/\md}\ge n$.
We also assume, without loss of generality, that
$t = O\left(n^{2-\gamma}p\right)$ for sufficiently small $\gamma>0$.
For two graphs $H_1,H_2$ denote by $H_1\cup H_2$ the graph whose vertex set is
$V\left(H_1\right)\cup V\left(H_2\right)$ and whose edge set is
$E\left(H_1\right)\cup E\left(H_2\right)$ (where multiple edges are ignored).
If $H_1,H_2$ are not vertex-disjoint we say they \emph{intersect} and denote
it by $H_1\sim H_2$.

\begin{lemma}\label{gnp:expectation}
  Let $H_1,H_2$ be two intersecting labelled copies of $H$ in $G$, and let 
  $H^*=H_1\cup H_2$.  Let $Z,Z^*$ count the number of appearances of a 
  copy of $H,H^*$ in $\Gamma_t$, respectively.  Then, \whp{},
  \begin{equation*}
    \E{Z^*\mid G} \ll \E^2{Z\mid G}.
  \end{equation*}
\end{lemma}

\begin{proof}
  According to \cref{key_corollary}, since $t\gg n$ and since $m_0\ge\ell/k$, 
  \whp{}
  \begin{equation*}
    \E{Z\mid G} = \Theta\left(n^{k-2\ell}t^\ell\right) = \omega(1),
  \end{equation*}
  and thus
  \begin{equation*}
    \E^2{Z\mid G} = \Theta\left(n^{2k-4\ell}t^{2\ell}\right).
  \end{equation*}
  Let $k',\ell'$ be the number of vertices and edges in the intersection
  $H_1\cap H_2$, respectively, and note that $H^*$ has $2k-k'$ vertices and
  $2\ell-\ell'$ edges.  We therefore have that, \whp{},
  \begin{equation*}
    \E{Z^*\mid G} = \Theta\left(
    n^{(2k-k')-2(2\ell-\ell')}t^{2\ell-\ell'} \right)
    = \Theta\left(
    n^{2k-k'-4\ell+2\ell'}t^{2\ell-\ell'}\right),
  \end{equation*}
  and thus
  \begin{equation*}
    \frac{\E^2{Z\mid G}}{\E{Z^*\mid G}}
    = \Theta\left(n^{k'-2\ell'}t^{\ell'}\right),
  \end{equation*}
  so, as $H_1,H_2$ are intersecting, either $\ell'=0$ and $k'>0$, in which
  case the above expression is $\omega(1)$, or $\ell'>0$, in which case
  $t^{\ell'}\gg n^{2\ell'-\ell'/\md}$ and the above expression is
  (since  $\md\ge \ell'/k'$),
  \begin{equation*}
    \omega\left(n^{k'-\ell'/\md}\right) = \omega(1).\qedhere
  \end{equation*}
\end{proof}

The following lemma shows that if two copies of $H$ are not vertex-intersecting,
then the events of their appearances in the trace are almost independent, in the
sense that their covariance is very small.

\begin{lemma}\label{gnp:nonintersecting}
  Let $H_1,H_2$ be two vertex-disjoint labelled copies of $H$ in $G$.
  Let $A_i$ be the event ``$H_i\subseteq\Gamma_t$'', and let $Z_i$ be its
  indicator, $i=1,2$.  Then \whp{}
  \begin{equation*}
    \cov\left(Z_i,Z_j\mid G\right)
    = o\left(t^{2\ell}n^{-4\ell}p^{-2\ell}\right).
  \end{equation*}
\end{lemma}

\begin{proof}
  According to \cref{key_lemma} and since $t\gg n$, \whp{},
  \begin{equation*}
    \pr{A_i\mid G}
    \sim \left(2t\right)^\ell\left(n^2p\right)^{-\ell},
  \end{equation*}
  and, since $H_1,H_2$ are vertex disjoint,
  \begin{equation*}
    \pr{A_1\land A_2\mid G}
    \sim \left(2t\right)^{2\ell}\left(n^2p\right)^{-2\ell},
  \end{equation*}
  and finally
  \begin{equation*}
    \pr{A_1\mid G}\cdot\pr{A_2\mid G}
    = \pr^2{A_i\mid G}
    \sim \left(2t\right)^{2\ell}\left(n^2p\right)^{-2\ell},
  \end{equation*}
  thus
  \begin{equation*}
    \cov\left(Z_i,Z_j\mid G\right)
    = o\left(t^{2\ell}n^{-4\ell}p^{-2\ell}\right).\qedhere
  \end{equation*}
\end{proof}

We now employ the second moment method to prove the positive part of the 
theorem.

\begin{proof}[Proof of the positive part of \cref{gnp:trace:subgraphs}]
  Let $Z$ count the number of copies of $H$ in $\Gamma_t$.
  Recall (e.g.\ from the proof of \cref{gnp:expectation}) that \whp{}
  \begin{equation*}
    \E{Z\mid G} = \Theta\left(n^{k-2\ell}t^\ell\right) = \omega(1).
  \end{equation*}
  
  Let $Y$ denote the number of copies of $H$ in $G$, and recall that \whp{} 
  $Y\sim \E{Y}$.
  Let $\mathcal{H}=\{H_1,H_2,\ldots H_Y\}$ be the set of all copies of $H$ in
  $G$, let $Z_i$ be the indicator of the event ``$H_i\subseteq \Gamma_t$'', let
  $\mathcal{U}$ be the set of all possible unions of two intersecting 
  (distinct) copies of $H$, and for $H^*\in\mathcal{U}$, let $Z_{H^*}$
  be the random variable counting the number of copies of $H^*$ in $\Gamma_t$.
  Write $i\sim j$ if $H_i\sim H_j$, and $i\nsim j$ otherwise.
  Since $|\mathcal{U}|=O(1)$, and using \cref{gnp:nonintersecting}, it follows 
  that, \whp{},
  \begin{align*}
    \var{Z\mid G}
    &= \sum_{i=1}^Y\sum_{j=1}^Y\cov\left(Z_i,Z_j\mid G\right)\\
    &= \sum_{i=1}^Y\var{Z_i\mid G}
      + \sum_{i\sim j}\cov\left(Z_i,Z_j\mid G\right)
      + \sum_{i\nsim j}\cov\left(Z_i,Z_j\mid G\right)\\
    &\le \sum_{i=1}^Y \E{Z_i\mid G}
      + \sum_{H_i\sim H_j}\pr{H_i\cup H_j\subseteq \Gamma_t\mid G}
      + o\left(n^{2k}p^{2\ell}\cdot t^{2\ell}n^{-4\ell}p^{-2\ell}\right)\\
    &= \E{Z\mid G}
      + 2\sum_{H^*\in\mathcal{U}}\E{Z_{H^*}}
      + o\left(\E^2{Z\mid G}\right)
    = o\left(\E^2{Z\mid G}\right).
  \end{align*}
  Chebyshev's inequality then yields the desired result.
\end{proof}

\section{Walking on \texorpdfstring{$K_n$}{Kn},
traversing trees}\label{section:kn}
Recall that $\td(G)$ denotes the minimum number of edge-disjoint trails in
$G$ whose union is the edge set of $G$.  In order to prove 
\cref{trace:subtrees}, we will prove the following theorem instead.

\begin{theorem}\label{trace:subtrees2}
  Let $T$ be a fixed tree on at least $2$ vertices with $\td(T)=\td$.
  Let $\Gamma_t$ be the trace of a random walk of length $t$ on $K_n$.
  Then,
  \begin{equation*}
    \lim_{n\to\infty} \pr{T\subseteq \Gamma_t}
    =\begin{cases}
      0 & t \ll n^{1-1/\td}\\
      1 & t \gg n^{1-1/\td}.
    \end{cases}
  \end{equation*}
\end{theorem}

The following lemma shows that \cref{trace:subtrees,trace:subtrees2} are in 
fact equivalent.
\begin{lemma}\label{trail:odd}
  For every connected $G$, $\td(G) = \max\left\{\odd(G)/2,1\right\}$.
\end{lemma}

\begin{proof}
  If $\odd(G)=0$ then $G$ is Eulerian, thus $\td(G)=1$.  Otherwise, let
  $\odd(G)=2k$, and let $v_1,v_2,\ldots,v_{2k}$ be the odd degree vertices.  
  Create $G'$ by adding the edges $\left\{v_{2i-1},v_{2i}\right\}$.  $G'$ is
  Eulerian; consider a tour (closed trail) $T$ in $G'$, and remove the added
  edges from that tour.  That creates exactly $k$ trails which make a
  partition of $E(G)$, thus $\td(G)\le k$.  On the other hand, every trail
  removed from $E(G)$ decreases $\odd(G)$ by at most $2$, hence $\td(G)\ge
  k$.
\end{proof}

\subsection{Proof of \texorpdfstring{\cref{trace:subtrees2}}
  {Theorem \ref{trace:subtrees2}}}
\label{proof:trace:subtrees2}
Throughout this section $T$ is a fixed non-empty tree with $k$ vertices,
$\ell=k-1$ edges and $\td(T)=\td$.

\subsubsection{Proof of the negative part}
\label{proof:trace:subtrees2:negative}
Assume $1\ll t\ll n^{1-1/\td}$.  Let $Z$ count the number of copies of $T$ in
$\Gamma_t$.  According to \cref{key_corollary},
\begin{equation*}
  \E{Z} = \Theta\left(
  n\sum_{r=\td}^{k-1}\left(\frac{t}{n}\right)^r \right).
\end{equation*}
Since $t\ll n$, we have that
\begin{equation*}
  \E{Z} = \Theta\left(n\left(\frac{t}{n}\right)^{\td}\right)
  = \Theta\left(n^{1-\td} t^{\td}\right)
  = o\left(n^{1-\td} n^{\td-1}\right)=o(1).
\end{equation*}
Markov's inequality then yields the result.\qed

\subsubsection{Proof of the positive part}
\label{proof:trace:subtrees2:positive}
We will need a couple of lemmas in order to prove the positive part of the
theorem.

\begin{lemma}\label{subtree}
  Let $T_1\subseteq T_2$ be two trees.  Then $\td(T_1)\le\td(T_2)$.
\end{lemma}

\begin{note}
The above lemma does not hold for $T_1,T_2$ which are not trees.
For example, the star $S_3$ with three leaves has $\td(S_3)=2$, but if
$G=S_3+e$ for any edge $e$ in the complement of $S_3$, then $\td(G)=1$.  
Similarly, the path $P_3$ of length $3$ has $\td(P_3)=1$, but $G=P_3-e$ where 
$e$ is the middle edge, is a forest with $\td(G)=2$.
\end{note}

\begin{proof}
  It suffices to show that every trail in $T_2$, restricted to the edges of
  $T_1$, is a trail in $T_1$.  Let $P$ be a trail in $T_2$.  Since $T_2$ is a
  tree, $P$ is a path.  Suppose to the contrary that the restriction of $P$
  to the edges of $T_1$, $P'$, is not a path.  Thus, it must have at least
  two connected components.  Let $u_1$ and $v_1$ be two vertices of $P'$
  which belong to two distinct connected components.  Thus in $T_2$ there are
  two distinct paths from $u_1$ to $v_1$, one which passes through $P$ and
  one which passes through $T_1$, in contradiction to the fact that $T_2$ is
  a tree.
\end{proof}

\begin{proof}[Alternative proof]
  In view of \cref{trail:odd} it suffices to show that
  $\odd(T_1)\le\odd(T_2)$, and this can be verified by starting with $T_1$
  and incrementally adding edges until reaching $T_2$, showing that each
  addition of an edge may not decrease the number of odd degree vertices.
\end{proof}

\begin{lemma}\label{ineq:tree}
  Let $T_1,T_2$ be two intersecting labelled copies of $T$ in $K_n$.  Let
  $k'$ and $\ell'$ denote the number of vertices and edges, respectively,
  of the intersection $T_1\cap T_2$, and let $\hat\td=\td(T_1\cup T_2)$.  Then
  \begin{equation*}
    k' - \ell' -2 + \hat\td/\td \ge 0.
  \end{equation*}
\end{lemma}

\begin{proof}
  Observe that $T_1\cap T_2$ is a forest.  If it is not a tree, then
  $k'-\ell' \ge 2$ and the claim follows.  Consider now the case where 
  $T_1\cap T_2$ is a tree.  In that case, $k'-\ell'=1$, thus it suffices to 
  show that $\hat\td\ge\td$.  Note that in that case it also follows that $T_1\cup
  T_2$ is a tree, since it is connected with $2k-k'$ vertices and
  $2\ell-\ell'$ edges, and
  \begin{equation*}
    (2k-k') - (2\ell-\ell') = 2(k-\ell) - (k'-\ell') = 1.
  \end{equation*}
  It follows that $T$ is a subtree of $T_1\cup T_2$, thus by \cref{subtree},
  $\td\le\hat\td$.
\end{proof}

In what follows, assume $n^{1-1/\td}\ll t$.  We also assume without loss of
generality that $t\ll n$.  The following lemma is the equivalent of 
\cref{gnp:expectation} for the case of traversing trees.

\begin{lemma}\label{tree:expectation}
  Let $T_1,T_2$ be two intersecting labelled copies of $T$ in $K_n$, and let
  $H=T_1\cup T_2$.  Let $Z,Z^*$ count the number of appearances of a copy of 
  $T,T^*$ in $\Gamma_t$, respectively.  Then
  \begin{equation*}
    \E{Z^*} \ll \E^2{Z}.
  \end{equation*}
\end{lemma}

\begin{proof}
  According to \cref{key_corollary} and since $n^{-1/\td}\ll t/n\ll 1$, we have 
  that
  \begin{equation*}
    \E{Z} = \Theta\left(n\left(\frac{t}{n}\right)^\td\right) = \omega(1),
  \end{equation*}
  and thus
  \begin{equation*}
    \E^2{Z} = \Theta\left(n^{2-2\td}t^{2\td}\right).
  \end{equation*}
  Write $\hat\td=\td(H)$.  Let $k',\ell'$ be the number of vertices and edges 
  of the intersection $T_1\cap T_2$, respectively.  Since $T_1\cap T_2$ is a 
  non-empty forest, $k'>\ell'$.  From \cref{key_corollary}, and since
  $t/n\ll 1$, we have that
  \begin{equation*}
    \E{Z^*} = \Theta\left(
    n^{2k-k'-(2\ell-\ell')}\left(\frac{t}{n}\right)^{\hat\td}\right)
    = \Theta\left(n^{2+\ell'-k'-\hat\td} t^{\hat\td}\right).
  \end{equation*}
  Now, if $\hat\td\ge 2\td$, then $(t/n)^{2\td-\hat\td} = \Omega(1)$ and
  \begin{equation*}
    \frac{\E^2{Z}}{\E{Y}}
    = \Theta\left(n^{k'-\ell'+\hat\td-2\td} t^{2\td-\hat\td}\right)
    = \Omega\left(n^{k'-\ell'}\right) = \omega(1).
  \end{equation*}
  On the other hand, if $\hat\td<2\td$, then $(t/n)^{2\td-\hat\td}\gg 
  n^{\hat\td/\td-2}$ and
  \begin{equation*}
    \frac{\E^2{Z}}{\E{Y}}
    = \Theta\left(n^{k'-\ell'+\hat\td-2\td} t^{2\td-\hat\td}\right)
    = \omega\left(n^{k'-\ell'-2+\hat\td/\td}\right),
  \end{equation*}
  and it follows from \cref{ineq:tree} that the last term is
  $\omega(1)$.
\end{proof}

Our next goal is to show that the events of the appearances of two 
vertex-disjoint graphs in the trace are not positively correlated.  To that aim,
we use a correlation inequality proved in \cite{Mc92}.
For finite non-empty sets $T$ and $V$, say that a collection $\mathcal{F}$ of
families $(W_v)_{v\in V}$ of subsets of $T$ is \emph{decreasing} if for every 
family $(W_v)_{v\in V} \in \mathcal{F}$, if $(W'_v)_{v\in V}$ satisfies 
$W'_v\subseteq W_v$ for every $v\in V$, then $(W'_v)_{v\in V} \in \mathcal{F}$.

\begin{lemma}[\cite{Mc92}*{Section 2}]\label{mcdiarmid}
  Let $T$ and $I$ be finite non-empty sets.  Let $I$ be partitioned into two
  non-empty sets $J$ and $K$.  Let $\mathcal{F}$ be a decreasing collection of
  families $(W_v)_{v\in J}$ and let $\mathcal{G}$ be a decreasing collection of
  families $(W_v)_{v\in K}$.  Let $(x_j)_{j\in T}$ be a family of independent
  random variables, each taking values in some set containing $I$, and, for each
  $v\in I$, let ${S_v=\{j\in T\mid x_j=v\}}$.  Let $F$ be the event
  ``$(S_v)_{v\in J}\in\mathcal{F}$'' and let $G$ be the event
  ``$(S_v)_{v\in K}\in\mathcal{G}$''.  In these settings,
  \begin{equation*}
    \pr{F\land G} \le \pr{F}\pr{G}.
  \end{equation*}
\end{lemma}

\begin{corollary}\label{tree:nonintersecting}
  Let $H_1,H_2$ be two vertex-disjoint subgraphs of $K_n$.  For $i\in[2]$, let 
  $A_i$ be the event ``$H_i\subseteq \Gamma_t$''.  Then $A_1,A_2$ are not 
  positively correlated.
\end{corollary}

\begin{proof}
  It is easy to verify that if two events are not positively correlated then so 
  are their complements.  It therefore suffices to prove that the complements 
  $B_1,B_2$ of $A_1,A_2$ are not positively correlated. 
  For $i\in[2]$ let $H_i=(V_i,E_i)$.
  We say that a family $(W_v)_{v\in V_i}$ of sets of times in 
  $\{0,1,\ldots,t\}$ \emph{misses an edge} $\{u,v\}\in E_i$ if there is no 
  $j\in[t]$ such that either $j-1\in W_u$ and $j\in W_v$ or $j-1\in W_v$ and 
  $j\in W_u$.
  Let $\mathcal{F},\mathcal{G}$ be the collections of all families of sets of 
  times which miss at least one edge from $E_1,E_2$, respectively, and observe 
  that $\mathcal{F},\mathcal{G}$ are decreasing.
  
  For $v\in V$, let $S_v$ be the (random) set of times at which the walk
  was located at $v$.  We can now write $B_1,B_2$ as the events
  ``$(S_v)_{v\in V_i}\in\mathcal{F}$'', ``$(S_v)_{v\in V_2}\in\mathcal{G}$'',
  respectively.
  Since $X_0,\ldots,X_t$ are independent, it follows from \cref{mcdiarmid}
  (with $J=V_1$, $K=V_2$, $T=\{0,\ldots,t\}$ and $x_j=X_j$)
  that
  $\pr{B_1\land B_2} \le \pr{B_1}\pr{B_2}$.
\end{proof}

\begin{proof}[Proof of the positive part of \cref{trace:subtrees2}]
  Recall that $n^{1-1/\td} \ll t\ll n$.
  Let $Z$ count the number of copies of $T$ in $\Gamma_t$.  Recall (e.g.\ 
  from the proof of \cref{tree:expectation}) that
  \begin{equation*}
    \E{Z} = \Theta\left(n\left(\frac{t}{n}\right)^\td\right) = \omega(1).
  \end{equation*}

  Let $\mathcal{T}=\left\{T_1,T_2,\ldots,T_y\right\}$ be the set of all
  copies of $T$ in $K_n$, let $Z_i$ be the indicator of the event
  ``$T_i\subseteq\Gamma_t$'', let $\mathcal{U}$ be
  the set of all possible unions of two intersecting (distinct) copies of
  $T$, and for $H\in\mathcal{U}$, let $Z_H$ be the random variable counting
  the number of copies of $H$ in $\Gamma_t$.  Write $i\sim j$ if $Z_i$ and
  $Z_j$ are positively correlated, and recall (from
  \cref{tree:nonintersecting}) that if $i\sim j$ then $T_i\sim T_j$ (that 
  is, $T_i,T_j$ intersect).  It follows that
  \begin{align*}
    \var{Z}
    &= \sum_{i=1}^y\sum_{j=1}^y\cov(Z_i,Z_j)\\
    &\le \sum_{i=1}^y \E{Z_i}
    + \sum_{i\sim j}\pr{T_i\cup T_j\subseteq \Gamma_t}\\
    &\le \E{Z} + \sum_{T_i\sim T_j}\pr{T_i\cup T_j\subseteq\Gamma_t}
    = \E{Z} + 2\sum_{H\in\mathcal{U}}\E{Z_H}.
  \end{align*}
  Since $|\mathcal{U}|=O(1)$, it follows from \cref{tree:expectation} that
  $\var{Z} = o\left(\E^2{Z}\right)$, and thus from Chebyshev's inequality
  it follows that $Z>0$ \whp{}.
\end{proof}

\subsection{Proof of \texorpdfstring{\cref{trace:subforests}}
  {Corollary \ref{trace:subforests}}}
\label{proof:trace:subforests}
Suppose first that $t\ll n^{1-2/\theta}$.  Let $i\in[z]$ such that
$\odd(T_i)=\theta$.  By \cref{trace:subtrees}, \whp{} $T_i$ is not a
subgraph of $\Gamma_t$, and hence $F$ is not a subgraph of $\Gamma_t$.

Now suppose that $t\gg n^{1-2/\theta}$.  We assume without loss of generality 
that $t\ll n$.
Let $s=\floor{t/z}$, and for $i\in[z]$ let $\Gamma_i$ be the trace restricted 
to the times $[(i-1)s,is-1)$.  For $i\in[z]$, let $A_i$ be the event 
``$T_i\subseteq \Gamma_i$'', and let $T_i'$ be the first copy of $T_i$ in 
$\Gamma_i$ (if there exists one; let it be an arbitrary tree otherwise).  Note 
that the events $A_i$ are mutually independent.  Let
\begin{equation*}
  U_i=\bigcup_{1\le j<i}V\left(T_j'\right),
\end{equation*}
let $B_i$ be the event that an edge from $\Gamma_i$ intersects $U_i$, and let 
$C_i=A_i\wedge \overline{B_i}$.  Observe that for $U\subseteq[n]$ with 
$|U|=O(1)$, the probability that an edge from $\Gamma_i$ intersects $U$ is 
$O(s|U|/n) = o(1)$.  It follows, using \cref{trace:subtrees}, that conditioning
on $C_1,\ldots,C_{i-1}$, the probability of $C_i$ is $1-o(1)$, and therefore,
\whp{}, the trace contains vertex-disjoint copies $T_1',\ldots,T_z'$ of 
$T_1,\ldots,T_z$, hence it contains a copy of $F$.\qed

\section{Concluding remarks and open problems}\label{section:concluding}
Our results give another confirmation to the assertion that random walks which
are long enough to typically cover a random graph, which is itself dense enough 
to be typically connected, leave a trace which ``behaves'' much like a random 
graph with a similar density.  On the other hand, at least on the complete 
graph, the results suggest that if the random walk is of sublinear length then 
it leaves a trace which is very different from a random graph with similar edge 
density.  In what other aspects do the two models differ?

In \cref{trace:subtrees} we have found, in particular, that a fixed path $P$ 
appears in the trace of a random walk on the complete graph \whp{} as long as 
$t\gg 1$.  In fact, it is not difficult 
to show that if $P$ is a path of length $\ell\ll\sqrt{n}$ and $t\ge \ell$, then 
$\Gamma_t$ contains a copy of $P$ \whp{}.  This is true since a random walk of 
length $t\ll\sqrt{n}$ typically does not intersect itself.  It may be 
interesting to find thresholds for the appearance of other ``large'' trees.  It
may also be interesting to find the threshold for the appearance of forests in
the trace of a random walk on a random graph.  Is it true, for example, that
if $p\ge n^{-1+\varepsilon}$ for some $\varepsilon>0$ then the thresholds are 
the same as in the case of $p=1$?  A slight variation in the proof of 
\cref{tree:expectation} works for random graphs as well, as long as
$\varepsilon\ge 1/\td$, but our use of \cref{mcdiarmid} already assumes that
the locations of the random walk are independent of each other.

Another possible direction would be to study the trace of the walk on other 
expander graphs, such as $(n,d,\lambda)$-graphs (see~\cite{KS06} for a survey), 
or on other random graphs, such as random regular graphs.  The small subgraph 
problem for random regular graphs of growing degree was settled by Kim, Sudakov and 
Vu~\cite{KSV07}.  They have shown that the degree threshold for the appearance 
of a copy of $H$ in a random regular graph is $n^{1-1/\md(H)}$, as long as $H$ 
contains a cycle.  Is it true that for ${d\ge n^{1-1/\md(H)+\varepsilon}}$, the 
time threshold for the appearance of $H$ in the trace of a random walk on a
random $d$-regular graph is also typically $n^{2-1/\md(H)}$, as in 
\cref{gnp:trace:subgraphs}?

\begin{acknowledgement}
  The authors wish to thank Alan Frieze, Asaf Nachmias and Yinon Spinka for 
  useful discussions, and two anonymous referees for their careful reading of 
  the paper and valuable comments and suggestions.
\end{acknowledgement}

\bibliography{library}
\end{document}